\documentclass{amsart}
\usepackage{amsmath,amsthm,amssymb,amsfonts, hyperref}
\usepackage{xcolor}

\newcommand{\seq}{\begin{equation*}}
\newcommand{\steq}{\end{equation*}}
\newcommand{\seqn}{\begin{equation}}
\newcommand{\steqn}{\end{equation}}
\usepackage{graphicx}

\newtheorem{theorem}{Theorem}[section]
\newtheorem{lemma}[theorem]{Lemma}

\numberwithin{equation}{section}

\theoremstyle{definition}

\title{The density of the graph of elliptic Dedekind sums}

\author[S. Bartell]{Stephen Bartell}
\email{sbartell@princeton.edu}
\address{Princeton University, 7903 Frist Campus Center, Princeton, NJ, 08544, USA}

\author[A. Halverson]{Abby Halverson}
\email{halver2@stolaf.edu}
\address{St. Olaf College, 1500 St. Olaf Ave, Northfield, MN, 55057, USA}

\author[B. Schlader]{Brenden Schlader}
\email{schlbre2@bvu.edu}
\address{Buena Vista University, 610 W. 4th St., Storm Lake, IA, 50588, USA}

\author[S. Truex] {Siena Truex}
\email{struex@middlebury.edu}
\address{14 Old Chapel Rd. 3682 Middlebury College, Middlebury, VT 05753, USA}

\author[T.A. Wong]{Tian An Wong}
\email{tiananw@umich.edu}
\address{University of Michigan-Dearborn, 4901 Evergreen Rd, 2002 CASL Building, Dearborn, MI 48128, USA}

\keywords{Elliptic Dedekind sums, complex continued fractions}
\subjclass[2020]{11F67   \and 11F41  }

\begin{document}

\begin{abstract}
We show that the graph of normalized elliptic Dedekind sums is dense in its image for arbitrary imaginary quadratic fields, generalizing a result of Ito in the Euclidean case. We also derive some basic properties of Martin's continued fraction algorithm for arbitrary imaginary quadratic fields.
\end{abstract}

\date{\today}
\maketitle
\tableofcontents \addtocontents{toc}{\protect\setcounter{tocdepth}{1}}

\section{Introduction}
\subsection{The density of elliptic Dedekind sums} 
The classical Dedekind sum $s(m,n)$ is defined for $m,n\in\mathbb{Z}$, $(m,n)=1$, $n\neq 0$, by
\[
s(m,n):=\frac{1}{4n}\sum_{k=1}^{n-1}\cot\left(\pi\frac{mk}{n}\right)\cot\left(\pi\frac{k}{n}\right).
\]
They are of arithmetic interest, arising originally from the transformation law of the Dedekind eta function. It has found applications to the special values of $L$-functions, and also areas outside of number theory such as knot theory, geometric topology, and combinatorics. Grosswald and Rademacher conjectured that the values of $s(m,n)$ are dense in $\mathbb R$, and moreover the graph $\{(m/n, s(m/n)): m/n\in\mathbb Q\}$ is dense in $\mathbb R^2$ \cite{radegross}, where $s(m/n) := s(m,n)$ whenever $(m,n)=1$. The latter statement (which implies the former) was first proved by Hickerson \cite{hick}.

Many generalizations of Dedekind sums have been made, and in this paper we study elliptic Dedekind sums. These are generalizations of classical Dedekind sums to complex lattices, or imaginary quadratic number fields.  Let $L$ be a non-degenerate lattice in $\mathbb C$.  We define
$$
E_k(z)=\sum_{\substack{x\in L,\\ x+z\not= 0}} (x+z)^{-k}|x+z|^{-s}\Big|_{s=0},
$$
where the value of the sum at $s=0$ is evaluated by means of analytic continuation.  Define  the ring of multipliers for $L$ as $\mathcal O_L=\{m \in \mathbb{C} :mL\subset L\}$. It is either equal to the ring of integers or to an order in an imaginary quadratic field.  Then, following Sczech \cite{sczech}, the elliptic Dedekind sums for $L$ are defined as
$$
D(a,c)=\frac{1}{c}\sum_{\mu \in L/cL} E_1\Big(\frac{a\mu}c\Big)E_1\Big(\frac{\mu}c\Big)
$$
for $a,c\in\mathcal O_L$, $c \neq 0$. 

Assume for the rest of this paper that $\mathcal O_L$ is the ring of integers $\mathcal O_K$ of an imaginary quadratic field $K=\mathbb Q(\sqrt{-D})$.  If $K$ has class number 1, one can again define $D(a/c) = D(a,c)$ for $a/c\in K$ and $a,c$ coprime in $\mathcal O_K$. (Note that if the class number is greater than 1, the fraction is no longer well-defined.)
Ito \cite{ito} showed that the graph 
\[
\{(a/c, \tilde D(a/c)): a/c\in K\}
\]
of the normalized elliptic Dedekind sum
$$
\tilde D(a/c) = \tilde D(a,c)=(i\sqrt{|d_K|}E_2(0))^{-1}D(a,c)
$$
is dense in $\mathbb C\times \mathbb R$ when $D=2,5,7$, where $d_K$ is the discriminant of $K$. (That is, when $\mathcal O_K$ is Euclidean and $D\neq 1,3$; note $\tilde D$ is trivial if and only if $D=1,3$.) It was also shown recently by the last author and others \cite{REU2021} that the image of $\tilde D(a,c)$ is dense in $\mathbb R$ for general lattices $L$ in $\mathbb C$ (and in particular for arbitrary class number)  using a recent method of Kohnen \cite{kohnen}.

\subsection{Main result}

In this paper, we extend these results to the density of the graph for imaginary quadratic fields of arbitrary class number.  Note that the reduced fraction $a/c$ associated to any $\alpha\in K$ is only uniquely defined when $K$ has class number 1. Also, the set of factorizations of an element in $\mathcal O_K$ is bounded \cite{martin2011nonunique}. For general class number, it will suffice for us to fix any nonunique representative $a/c$ of $\alpha$ and consider $\tilde  D(\alpha) = \tilde D(a,c)$. Note that $\tilde D(a,c)= \tilde D(\lambda a,\lambda c)$ for any nonzero $\lambda\in \mathcal O_K$ by \cite[(15)]{sczech}. With this convention, we state our main result.

\begin{theorem}
\label{main}
Let $K = \mathbb{Q}(\sqrt{-D})$ with $D\neq 1,3$. Then the graph of the normalized elliptic Dedekind sum 
\begin{equation}
\label{set}
\{(\alpha, \tilde D(\alpha)) : \alpha\in K\}
\end{equation}
is dense in $\mathbb C\times \mathbb R$.
\end{theorem}

\noindent 
Following the method of Ito and Hickerson, our proof first relies on a generalization of the continued fraction algorithm recently established by Martin \cite{martin2023continued} for imaginary quadratic fields of general class number. 
The second ingredient in our proof is Sczech's homomorphism $\Phi : SL_2(\mathcal{O}_K) \rightarrow \mathbb{C}^{+}$ \cite{sczech} given by
$$\Phi
\begin{pmatrix}
a & b \\
c & d
\end{pmatrix} = 
\begin{cases}
        E_2(0)I\left(\dfrac{a+d}{c}\right)-D(a,c), & c \not = 0\\
        E_2(0)I\left(\dfrac{b}{d}\right), & c=0
    \end{cases}
$$
where $I(z) := z - \bar{z} = 2 \operatorname{Im}(z)$. It was extended to $GL_2(\mathcal O_K)$ by Obaisi \cite{obaisi2000eisenstein} as follows: for a more general matrix $A \in GL_2(\mathcal O_K)$, we have 
\[
\Phi
\begin{pmatrix}
a & b \\
c & d
\end{pmatrix} = 
    \begin{cases}
        E_2(0)I\left(\dfrac{a+\det(A)d}{c}\right)-D(a,c), & c \not = 0\\
        E_2(0)I\left(\dfrac{b}{d}\right), & c=0
    \end{cases}
\]
by evaluating \cite[(4.2)]{obaisi2000eisenstein} at the point $u = (0,0)$.
\footnote{Actually, Obaisi's generalization of the Sczech cocycle in \cite{obaisi2000eisenstein} should be multiplied by $-1$, but this does not affect the computations.}

Let $A \in GL_{2}(\mathcal{O})$ act on the extended complex plane by M\"obius transformations. If $d \not= 0$, then we see $A(0) = \frac{b}{d}$. If $c \not = 0$, we also have $A(\infty) = \frac{a}{c}$ and $$A^{-1}(\infty) = \begin{pmatrix}
\frac{d}{\det(A)} & \frac{-b}{\det(A)} \\
\frac{-c}{\det(A)} & \frac{a}{\det(A)}
\end{pmatrix}(\infty) 
= -\frac{d}{c}.
$$
Note that $c$ and $d$ cannot both be zero because $A \in GL_{2}(\mathcal{O})$. In addition, $c = 0$ if and only if $A(\infty) = \infty$, and $d = 0$ if and only if $A(\infty) \not= \infty$.
Therefore, we can rewrite the homomorphism extension as
\begin{equation}
\label{phi}
\Phi(A) = 
\begin{cases}
        E_2(0)I(A(\infty) - \det(A)(A^{-1}(\infty)))-D(A(\infty)), & A(\infty) \not = \infty\\
        E_2(0)I(A(0)), & A(\infty) = \infty
    \end{cases}.
\end{equation}
This is the form of the homomorphism that we shall use. 

A natural question to ask is whether the graph of normalized elliptic Dedekind sums is equidistributed. It would be natural to expect such a result based on the earlier works \cite{KLW} and \cite{myerson1988dedekind}. We also note that the conjecture of Ito in the same paper \cite[\S3]{ito} on the bias of Dedekind sums was recently proved in the case of classical Dedekind sums \cite{minelli2022bias}; it would be interesting to explore the conjecture for elliptic Dedekind sums.

We conclude with a brief summary of the contents of this paper. In Section \ref{algsec}, we recall Martin's continued fraction algorithm. We also derive some properties of Martin's continued fraction algorithm generalizing the properties of the Hurwitz continued fraction algorithm in \cite{ito} and classical continued fractions \cite{hick}. 
In Section \ref{approx} we prove an approximation property for the generalized continued fractions analogous to \cite[Lemma 1]{ito}. In Section \ref{proof} we prove the main theorem.

\section{Martin's algorithm}
\label{algsec}

Martin \cite{martin2023continued} provides a continued fraction algorithm that can be executed in an arbitrary imaginary quadratic field $K$ with ring of integers $\mathcal{O} = \mathcal{O}_K$. For any $z \in \mathbb{C}\setminus K$, terminating the algorithm after $n$ iterations produces an approximation $\frac{p_{n}}{q_{n}} \in K$ of $z$. The algorithm is implemented as follows: For $a, b \in \mathbb{C}$, let
$$S(a,b)=\begin{pmatrix}
a & 1 \\
b & 0
\end{pmatrix}.$$
Fix $\varepsilon\in (0,1)$ Let $B \subset \mathcal{O}\backslash\{0\}$ be a finite admissible subset for $varepsilon$ in the sense of \cite[Definition 2.4]{martin2023continued}. Intuitively, an admissible set is taken to be a set $B$ of denominators that are enough so that the collection of discs $D(a/b,\varepsilon/|b|)$ with $a\in K$ cover the complex plane. By \cite[Theorem 4.3]{martin2023continued}, we may take $B = \{1,2,\dots , \lfloor \sqrt{|d_K|}\rfloor\}$. 

 We define recursively as in  Algorithm 1 in \cite[\S2.2]{martin2023continued} the sequence of matrices
$$
M_{0} = \begin{pmatrix}
1 & 0 \\
0 & 1
\end{pmatrix},\qquad M_{n} = M_{n-1}S\bigg(\frac{a_{n}}{b_{n-1}},\frac{b_{n}}{b_{n-1}}\bigg),\ n \geq 1,
$$
where the coefficients $a_{i} \in \mathcal{O}$, $b_{i} \in B$ can be determined by Algorithm 2 in \cite[\S3.2]{martin2023continued}.
Note that we are following Martin's convention in Algorithm 1 in \cite[\S2.2]{martin2023continued} that $b_{0}=1$. 

We then define the $n$-th convergents $p_{n}$ and $q_{n}$ as the left column entries of $M_{n}$. By the definition of $S$, the right column entries of $M_{n}$ are then $p_{n-1}$ and $q_{n-1}$. Hence, we have
$$
M_{n} = \begin{pmatrix}
p_{n} & p_{n-1} \\
q_{n} & q_{n-1}
\end{pmatrix},
$$
and viewing $M_{n}$ as a M\"{o}bius transformation on the extended complex plane, we can write $M_{n}(\infty) = \frac{p_{n}}{q_{n}}$. It follows inductively from the definition of $S$ that 
\begin{equation}
\label{21}
    \det M_{n} = p_{n}q_{n-1} - p_{n-1}q_{n} = (-1)^{n}b_{n}.
\end{equation}
Martin verifies that the $n$-th approximation $\frac{p_{n}}{q_{n}}$ can be written in the continued fraction form
$$
\frac{p_n}{q_n} = \frac{a_1}{b_1}+\cfrac{b_{0}/b_{1}}{\frac{a_2}{b_2}+\cfrac{b_{1}/b_{2}}{\ddots +\frac{a_{n-1}}{b_{n-1}}+\frac{b_{n-2}/b_{n-1}}{a_{n}/b_{n}}}}.
$$

\subsection{Properties}
\label{app}

We next recall the properties of Martin's continued fraction algorithm proved in \cite{martin2023continued} that we shall require. Define $\mu = \max_{B}|b|$ for a fixed finite admissible set $B$ with $\varepsilon \in (0,1)$. Then for any $n\ge 1$, the following properties hold.

\begin{enumerate}
\item (Lemma 3.1)
 \begin{equation}
 \label{31}
 |z_{n}| \geq 1/\varepsilon.
 \end{equation}

\item (Proposition 3.2) 
    \[
    |q_{n}z-p_{n}| \leq \varepsilon|q_{n-1}z-p_{n-1}|.
    \]
\item 
(Theorem 3.11)
If $0 \leq n' < n$, then
\begin{equation}
\label{311}
|q_{n}| > \frac{(1-\varepsilon^{2})^{2}|q_{n'}z_{n'}|}{4\varepsilon^{n-n'}\mu^{2}}.
\end{equation}
where $z_n = \dfrac{q_{n-1}z-p_{n-1}}{p_n-q_nz}.$
In particular, $|q_{n}| > (1-\varepsilon^{2})^{2}/4\varepsilon^{n}\mu^{2}$. This implies $\lim_{n\to \infty}|q_{n}| = \infty$.
    \item 
    (Corollary 3.12)
For all $n\ge1$,
\begin{equation}
\label{312}
\bigg|z-\frac{p_{n}}{q_{n}}\bigg| < \frac{4\varepsilon^{2n}\mu^{2}}{(1-\epsilon^{2})^{2}},
\end{equation}
which implies that $\lim_{n\to \infty}|z-\frac{p_n}{q_n}| = 0$.
\end{enumerate}

We next derive several additional identities for Martin's generalized continued fractions. They are not used for the main theorem, but they were developed in the course of its proof. We include them here as they establish the general analogues of the Euclidean case. Let $a_n \in \mathcal{O}$ and $b_n \in B$ as before. Now define recursively the following notation
$$[b_0]=b_0, \qquad [b_0,a_1,b_1]=a_1$$
$$[b_0,\dots, a_n, b_n] = \frac{a_n}{b_{n-1}}[b_0,\dots , a_{n-1}, b_{n-1}]+ \frac{b_n}{b_{n-1}}[b_0,\dots, a_{n-2}, b_{n-2}].$$
The first property gives a continued fraction expansion identity for the convergents.

\begin{lemma}
The convergents $p_n$ and $q_n$ from Martin's algorithm can be written as
$$p_n=[b_0, a_1, b_1, a_2, \dots,a_{n-1},b_{n-1}, a_n, b_n]$$
$$q_n=[b_1, a_2, b_2, a_3, \dots,a_{n-1}, b_{n-1}, a_n, b_n].$$
\end{lemma}


\begin{proof}
To prove these statements we will use induction. Beginning with $p_n$, we use two base cases where $n=0$ and $n=1$. When $n=0$, we have that $p_0=1$, and since we assume that $b_0=1$, $p_0=1=b_0=[b_0]$. Additionally when $n=1$, we have that $p_1=a_1=[b_0,a_1,b_1]$. We then assume that the $n-1$ and $n-2$ cases hold, and examine the definition of $[b_0, a_1, b_1, a_2, \dots,a_{n-1},b_{n-1}, a_n, b_n]$. We have that 
\begin{align*}
&[b_0, a_1, b_1, a_2, \dots,a_{n-1},b_{n-1}, a_n, b_n]\\
&= \frac{a_n}{b_{n-1}}[b_0,\dots , a_{n-1}, b_{n-1}]+ \frac{b_n}{b_{n-1}}[b_0,\dots, a_{n-2}, b_{n-2}] \\&= \frac{a_n}{b_n}p_{n-1}+\frac{b_n}{b_{n-1}}p_{n-2}=p_n,
\end{align*}
which proves the equality for $p_n$. 

Next we do the same thing for $q_n$, instead using base cases $n=1$ and $n=2$. When $n=1$, we have that $q_1=b_1=[b_1]$. When $n=2$, we have that $q_2=a_2=[b_1,a_2,b_2]$. Now we assume that the $n-1$ and $n-2$ cases hold, and have by definition and the assumption of the $n-1$ and $n-2$ cases by a similar computation that 
$
[b_1, a_2, b_2, a_3, \dots,a_{n-1},b_{n-1}, a_n, b_n]=q_n.
$
\end{proof}

The following properties then generalize those of Hurwitz continued fractions, as stated in \cite[\S2]{ito}. These properties are analogous to those of Hurwitz continued fractions, with the differences that our coefficients are no longer necessarily in the ring of integers, that $b_n$ is now part of the third property, and that the index is increased by one in \eqref{A3} because we start at $b_0$ without an $a_0$.\footnote{We thank the reviewer for also pointing out to us that these hold as formal identities in the field of fractions of $\mathbb Z$ adjoined with sufficiently many variables.}

\begin{lemma}
The following identities hold:
\begin{equation}
\label{A1}
    [b_0,\dots,a_n, b_n] = \frac{a_1}{b_1}[b_1, \dots, a_n, b_n] +\frac{b_0}{b_1}[b_2, \dots, a_n, b_n],
\end{equation}
\begin{equation}
\label{A2}
        [b_0,a_1, b_1, \dots,a_n, b_n]=[b_n,a_n, \dots, b_1, a_1,b_0],
\end{equation}
\begin{equation}
\label{A3}
[b_m,\dots, b_n][b_{m+1},\dots, b_{n-1}]-[b_{m}, \dots, b_{n-1}][b_{m+1},\dots, b_n] =b_nb_m(-1)^{n-m}.
\end{equation}
\end{lemma}

\begin{proof}
Beginning with the bracket notation for $\frac{p_n}{q_n}$, one checks that 
\begin{equation*}
    \frac{[b_0,a_1,\dots,a_n,b_n]}{[b_1,a_2,\dots,a_n,b_n]}= \frac{a_1}{b_1} + \frac{b_0}{b_1}\cdot \frac{[b_2,a_3,\dots,a_n,b_n]}{[b_1,a_2,\dots,a_n,b_n]}.
\end{equation*}
Multiply the left and rightmost sides of this equality to obtain
\begin{equation*}
    [b_0,\dots,a_n, b_n] = \frac{a_1}{b_1}[b_1, \dots, a_n, b_n] +\frac{b_0}{b_1}[b_2, \dots, a_n, b_n].
\end{equation*}
This proves the first identity \eqref{A1}.

We prove the second property by induction. To start, we establish the base cases of $n=0$ and $n=1$. For $n=0$, we have that $[b_0]=b_0=[b_0]$. For $n=1$ we have that $[b_0,a_1,b_1]=a_1=[b_1,a_1,b_0]$. Having established the base cases, we assume that the $n-1$ and $n-2$ cases hold. 
We have that 
$$[b_0,a_1,\dots, a_n,b_n]= \frac{a_1}{b_1}[b_1,a_2,\dots,a_n,b_n]+\frac{b_0}{b_1}[b_2,a_3,\dots,a_n,b_n]$$ by (2), and because we are assuming the $n-1$ and $n-2$ cases, this is equal to
\begin{align*}
\frac{a_1}{b_1}[b_n,a_n,\dots,a_2,b_1]+\frac{b_0}{b_1}[b_n,a_n,\dots,a_3,b_2]
&=[b_n,a_n,\dots,a_1,b_0].
\end{align*}
We prove \eqref{A3} both when $m$ is fixed and when $n$ is fixed. When we fix $m$, we have by direct computation that 
\begin{multline*}
    [b_{m+1},a_m,\dots,a_{n-1},b_{n-1}]\left(\frac{a_n}{b_{n-1}}[b_m,a_{m+1},\dots,a_{n-1},b_{n-1}]+\frac{b_n}{b_{n-1}}[b_m,a_{m+1},\dots,a_{n-2},b_{m-2}] \right)
    \\-[b_m,a_{m+1},\dots,a_{n-1},b_{n-1}]\left(\frac{a_n}{b_{n-1}}[b_{m+1},a_{m+2},\dots,a_{n-1},b_{n-1}]+ \frac{b_n}{b_{n-1}}[b_{m+1},a_{m+2},\dots,a_{n-2},b_{n-2}]\right)
    \end{multline*}
    equals
    \[
    \frac{b_n}{b_{n-1}}\left((-1)^{(n-1)-m}b_{n-1}b_m\right)=(-1)^{(n-m)}b_nb_m.
    \]
When we fix $n$, we have similarly that  
\begin{multline*}
    [b_{m+1},a_m,\dots,a_{n-1},b_{n-1}]\left(\frac{a_{m+1}}{b_{m+1}}[b_{m+1},a_{m+2},\dots,a_{n},b_{n}+\frac{b_m}{b_{m+1}}[b_{m+2},a_{m+3},\dots,a_{n},b_{n}] \right)
    \\-[b_{m+1},a_{m+2},\dots,a_{n},b_{n}]\left(\frac{a_{m+1}}{b_{m+1}}[b_{m+1},a_{m+2},\dots,a_{n-1},b_{n-1}]+ \frac{b_m}{b_{m+1}}[b_{m+2},a_{m+3},\dots,a_{n-1},b_{n-1}]\right)
\end{multline*}
equals
\[
\frac{b_m}{b_{m+1}}\left((-1)^{n-(m+1)}b_{m+1}b_n\right)=(-1)^{(n-m)}b_nb_m.
\]
\end{proof}

\section{Approximation by generalized continued fractions}
\label{approx}

We now prove an approximation result generalizing \cite[Lemma 1]{ito} to the convergents produced by Martin's algorithm. Denote
$$\zeta = \frac{(1-\varepsilon^2)^2}{4\varepsilon^2\mu^2},$$ where $\varepsilon$ and $\mu$ are defined as above.
\begin{lemma}
\label{approxlem}
 Consider the continued fraction expansion of $z\in \mathbb C \backslash K$. Let $\delta \in (0, \zeta)$ and $W \in GL_2(\mathcal{O})$ be such that $W(\infty) \not= \infty$ and $|W(\infty)| \geq \frac{1}{\zeta - \delta}$. Then, $M_{n}W(\infty) \not= \infty$, and for any $\epsilon>0$, there exists an N $\in \mathbb{N}$  independent of $W$ such that 
$$|z-M_{n}W(\infty)| < \epsilon$$ for all $n \geq N$.
\end{lemma}

\begin{proof}
Since $0 < \frac{1}{\zeta - \delta} \leq |W(\infty)| < \infty$, we can define $w := \frac{1}{W(\infty)}$. Note that $|w|\leq \zeta-\delta$ by the assumption of the lemma. In addition, we have
\begin{align*}
    M_{n}W(\infty) &= \begin{pmatrix}
p_{n} & p_{n-1} \\
q_{n} & q_{n-1}
\end{pmatrix}(W(\infty))\\
    &= \frac{p_{n}W(\infty) + p_{n-1}}{q_{n}W(\infty) + q_{n-1}}\\
    &= \frac{p_{n} + \frac{1}{W(\infty)} p_{n-1}}{q_{n} + \frac{1}{W(\infty)}q_{n-1}}\\
    &= \frac{p_{n} + wp_{n-1}}{q_{n} + wq_{n-1}}.
\end{align*}

In order to show that $M_{n}W(\infty)\not=\infty$, we must show that $q_{n} + wq_{n-1} \not= 0$. Setting $n' = n-1$ in \eqref{311} and combining with \eqref{31}, we see that for $n \geq 2$,
$$|q_n|>\frac{(1-\varepsilon^2)^2}{4\varepsilon\mu^2}|q_{n-1}z_{n-1}| \geq \frac{(1-\varepsilon^2)^2}{4\varepsilon^2\mu^2}|q_{n-1}| = \zeta|q_{n-1}|.
$$
Using the triangle inequality, we have 
\begin{align*}
|q_n+wq_{n-1}|&\geq |q_n|-|wq_{n-1}|\\
&>\zeta|q_{n-1}| - |w||q_{n-1}|\\
&\geq \delta|q_{n-1}|.
\end{align*}
Since $|q_{n}+wq_{n-1}| > \delta|q_{n-1}|>0$, we can conclude that $q_{n}+wq_{n-1} \not= 0$, as desired.

Moving on to the second part of the proof, we must show that there exists $N \in \mathbb{N}$ such that
$$|z-M_{n}W(\infty)| = \bigg|z-\frac{p_{n} + wp_{n-1}}{q_{n} + wq_{n-1}}\bigg| < \epsilon
$$
for all $n\ge N$.

Applying the triangle inequality, we have
$$\left|z - \frac{p_n + wp_{n-1}}{q_n+wq_{n-1}}\right| \leq \left|z - \frac{p_n}{q_n}\right| + \left|\frac{p_n}{q_n} - \frac{p_n + wp_{n-1}}{q_n+wq_{n-1}}\right|.$$ 
By \eqref{312}, the first term above tends to zero, so we can define $N_{1} \in \mathbb{N}$ such that for all $n \geq N_{1}$, we have
$$
\left|z - \frac{p_n}{q_n}\right| < \frac{\epsilon}{2}.
$$
Now we are left to bound the second term above. By combining fractions, we see that
\begin{align*}
\left|\frac{p_n}{q_n}- \frac{p_n + wp_{n-1}}{q_n+wq_{n-1}}\right| 
&=\left|\frac{p_{n}q_{n}+wp_nq_{n-1}-p_{n}q_{n}-wp_{n-1}q_n}{q_n(q_n+wq_{n-1})}\right| \\
&=\left|\frac{w(p_nq_{n-1}-p_{n-1}q_n)}{q_n(q_n+wq_{n-1})}\right|.
\end{align*}
By \eqref{21}, we see that
\begin{align*}
\left|\frac{w(p_nq_{n-1}-p_{n-1}q_n)}{q_n(q_n+wq_{n-1})}\right|
&= \left|\frac{wb_n(-1)^{n}}{q_n(q_n+wq_{n-1})}\right| = \left|\frac{wb_n}{q_n(q_n+wq_{n-1})}\right|.
\end{align*}
Finally, using that $|q_{n}+wq_{n-1}| > \delta|q_{n-1}|$, we have
\begin{align*}
 \left|\frac{wb_n}{q_n(q_n+wq_{n-1})}\right| 
&< \left|\frac{wb_n}{\delta q_n q_{n-1}}\right| \leq \frac{\mu|w|}{|\delta|}\frac{1}{|q_{n}q_{n-1}|}.
\end{align*}
Since the rightmost side of this inequality tends to 0  as $n$ approaches infinity by \eqref{311}, there exists $N_2 \in \mathbb{N}$ such that for all $n \geq N_2$, we have
$$
\left|\frac{p_n}{q_n} - \frac{p_n + wp_{n-1}}{q_n+wq_{n-1}}\right| < \frac{\mu|w|}{|\delta|}\frac{1}{|q_{n}q_{n-1}|}< \frac{\epsilon}{2}.
$$
Letting $N=$ max$\{N_1,N_2\}$ completes the proof.
\end{proof}

\section{Proof of density}
\label{proof}

We now prove the main theorem. 
We utilize the homomorphism \eqref{phi} to prove density in the general case. Following along with Ito's proof of \cite[Theorem 2]{ito}, we will show that the set
$$\{(x, (\sqrt{|d|}i)^{-1}I(x-z)): x, z \in \mathbb{C}-K\}$$
is contained in the closure of the set
$$\{(\alpha, \tilde{D}(\alpha)): \alpha \in K\}.
$$
Recall from the discussion in the introduction that any element of $K$ can be written as $a/c$ with $a,c\in \mathcal O_K$. This representation is non-unique but there are only finitely many choices, and here we are fixing a choice of representative. Our proof is independent of this choice.  

Let $\epsilon>0$ and $x, z \in \mathbb{C}-K$. 
Suppose first that for $|\alpha - x|< \epsilon,$ we have
\begin{equation}
    \label{dalph}
|\tilde{D}(\alpha)-(\sqrt{|d|}i)^{-1}I(x-z)| \leq 4(\sqrt{|d|})^{-1}\epsilon.
\end{equation}
It then follows that 
\[
(x, (\sqrt{|d|}i)^{-1}I(x-z))\in \overline{\{(\alpha, \tilde{D}(\alpha)): \alpha \in K\}},\]
since that $K$ is dense in $\mathbb C$ (see for example \cite{Oswald}).   It thus remains to prove \eqref{dalph}. We shall do this by constructing an $A\in GL_2(\mathcal O)$ such that $\alpha=A(\infty)$, which approximates $x$, and $\beta=A^{-1}(\infty)$ approximates $z$. Then \eqref{dalph} will follow from evaluating $\Phi(A)$ in two ways. 

By Lemma \ref{approxlem}, we can choose sufficiently large $m, n \in \mathbb{N}$ such that for any $W \in GL_2(\mathcal{O})$ with $|W(\infty)| \geq \frac{2}{\zeta}$, we have
$$|x-M_{m, x}W(\infty)| < \epsilon$$
and
$$|z-M_{n, z}W(\infty)| < \epsilon,$$
where $M_{m, x}$ denotes the $m$-th convergent matrix in the continued fraction representation of $x$ and similarly for $M_{n, z}$.

Let $S = M_{m, x}^{-1}M_{n, z}$, and let  $S^{*}$ be given by
$$S^{*} = 
\begin{cases}
        S, & |S(\infty)| \not = \infty\\
        \begin{pmatrix}
0 & -1 \\
1 & 0
\end{pmatrix}S, & |S(\infty)| = \infty.
    \end{cases}
$$
We claim that $|S^{*}(\infty)| \not = \infty$. First suppose that $|S(\infty)| \not = \infty$. In this case, we simply have $|S^{*}(\infty)| = |S(\infty)| \not = \infty$. Next suppose $|S(\infty)| = \infty$. In this case, we have
$$
|S^{*}(\infty)| = \bigg|\begin{pmatrix}
0 & -1 \\
1 & 0
\end{pmatrix}S(\infty)\bigg| = \bigg|\begin{pmatrix}
0 & -1 \\
1 & 0
\end{pmatrix}(\infty)\bigg| = 0 \not = \infty.
$$
In both cases, this also tells us that $|S^{*^{-1}}(\infty)| \not = \infty$ because $|S^{*^{-1}}(\infty)| = \infty$ would imply $|S^{*}(\infty)| = \infty$, a contradiction.

We also note that in either case  we have 
$$\Phi(S^{*}) = \Phi(S).$$ In the first case, this is clear because $S^{*} = S$. In the second case, we use the fact that $\Phi$ is a homomorphism to see that
$$\Phi(S^{*}) = \Phi\bigg(\begin{pmatrix}
0 & -1 \\
1 & 0
\end{pmatrix}S\bigg) = \Phi\begin{pmatrix}
0 & -1 \\
1 & 0
\end{pmatrix} + \Phi(S) = \Phi(S),$$
since 
$$4\Phi\begin{pmatrix}
0 & -1 \\
1 & 0
\end{pmatrix} = \Phi\left(\begin{pmatrix}
0 & -1 \\
1 & 0
\end{pmatrix}^4\right) = \Phi(I) = 0.$$
Finally, we note that $\det(S^{*}) = \det(S)$.

Now given $u \in \mathcal{O}$, we can define the matrix $A \in GL_2(\mathcal{O})$ by
$$A = M_{m, x}T^{u}S^{*}T^{-u}M_{n,z}^{-1},
\qquad 
T^{u} = \begin{pmatrix}
1 & u \\
0 & 1
\end{pmatrix},\ T^{-u} = \begin{pmatrix}
1 & -u \\
0 & 1
\end{pmatrix}.
$$
Using the property that $\det(S^*) = \det(S) = \det(M_{m, x}^{-1}M_{n, z})$, we can verify by direct computation that $\det(A) = 1$.
Next, let 
\[
\alpha = A(\infty),\qquad \beta = A^{-1}(\infty) = M_{n,z}T^{u}S^{*^{-1}}T^{-u}M^{-1}_{m,x}(\infty)
\]
and
\[
W_{1} = T^{u}S^{*}T^{-u}M^{-1}_{n, z},\qquad W_{2} = T^{u}S^{*^{-1}}T^{-u}M^{-1}_{m,x},
\]
so that 
\[
\alpha = M_{m,x}W_{1}(\infty),\qquad \beta = M_{n,z}W_{2}(\infty).
\]
We now pick $u \in \mathcal{O}$ sufficiently large so that 
\[
|W_{1}(\infty)| \geq \frac{2}{\zeta},
\qquad |W_{2}(\infty)| \geq \frac{2}{\zeta}.
\]
Note that $\frac{2}{\zeta}$ can be written as $\frac{1}{\zeta-\delta}$ for $\delta = \frac{\zeta}{2} \in (0, \zeta)$. 

We must make sure that we are always able to pick a coefficient $u$ with this property. Let $S^{*}=\begin{pmatrix}
s_{1} & s_{2} \\
s_{3} & s_{4}
\end{pmatrix}$. Note that $ M^{-1}_{n, z}=\begin{pmatrix}
q_{n-1} & -p_{n-1} \\
-q_{n} & p_{n}
\end{pmatrix}$, where these are convergents of $z$ using Martin's algorithm. Before proceeding with any computation, we add the simple condition that $u \not= \frac{s_{4}q_{n}-s_{3}q_{n-1}}{s_{3}q_{n}}$, noting that $q_{n}\not=0$ and $s_{3}\not= 0$ since $S^{*} \in GL_2(\mathcal{O})$ and $S^{*}(\infty) \not = \infty$. This condition implies that $s_{3}(q_{n-1}+q_{n}u)-s_{4}q_{n}\not=0$ and, as we will see below, ensures that $W_{1}(\infty)\not=\infty$. Now, we can see that
\begin{align*}
    W_{1} &= T^{u}S^{*}T^{-u}M^{-1}_{n, z}\\
&= \begin{pmatrix}
1 & u \\
0 & 1
\end{pmatrix}\begin{pmatrix}
s_{1} & s_{2} \\
s_{3} & s_{4}
\end{pmatrix}\begin{pmatrix}
1 & -u \\
0 & 1
\end{pmatrix}
\begin{pmatrix}
q_{n-1} & -p_{n-1} \\
-q_{n} & p_{n}
\end{pmatrix}\\
&= \begin{pmatrix}
s_{1}+s_{3}u & s_{2}+s_{4}u \\
s_{3} & s_{4}
\end{pmatrix} \begin{pmatrix} 
q_{n-1} + q_{n}u & -p_{n-1} - p_{n}u \\
-q_{n} & p_{n}
\end{pmatrix}\\
&= \begin{pmatrix}
(s_{1}+s_{3}u)(q_{n-1}+q_{n}u) - (s_{2}+s_{4}u)q_{n} & (s_{1}+s_{3}u)(-p_{n-1} - p_{n}u)+(s_{2}+s_{4}u)p_{n} \\  
s_{3}(q_{n-1}+q_{n}u) - s_{4}q_{n} & s_{3}(-p_{n-1} - p_{n}u)+s_{4}p_{n}
\end{pmatrix}.
\end{align*}
We can use this, along with the triangle inequality, to see that
\begin{align*}
|W_{1}(\infty)| &= \bigg|\frac{(s_{1}+s_{3}u)(q_{n-1}+q_{n}u) - (s_{2}+s_{4}u)q_{n}}{s_{3}(q_{n-1}+q_{n}u) - s_{4}q_{n}}\bigg|\\
&=\bigg|\frac{u(s_{3}(q_{n-1}+q_{n}u) - s_{4}q_{n}) + s_{1}(q_{n-1}+q_{n}u) - s_{2}q_{n}}{s_{3}(q_{n-1}+q_{n}u) - s_{4}q_{n}}\bigg|\\
&=\bigg|u+\frac{s_{1}(q_{n-1}+q_{n}u) - s_{2}q_{n}}{s_{3}(q_{n-1}+q_{n}u) - s_{4}q_{n}}\bigg|\\
&\geq |u| - \bigg|\frac{s_{1}(q_{n-1}+q_{n}u) - s_{2}q_{n}}{s_{3}(q_{n-1}+q_{n}u) - s_{4}q_{n}}\bigg|.
\end{align*}
As $|u|$ tends to infinity, the second term in the inequality above tends to 
$$-\left|\frac{s_{1}}{s_{3}}\right| = -|S^{*}(\infty)|,$$
which, by definition of $S^{*}$, is not infinite. Therefore, as $|u|$ tends to infinity, $|W_{1}(\infty)|$ also tends to infinity. We can use a very similar argument, which uses the $m$-th convergents of $x$ and the fact that $S^{*^{-1}}(\infty) \not = \infty$, to show that $|W_{2}(\infty)|$ can also be made arbitrarily large as $|u|$ becomes arbitrarily large. Again, we can impose a simple condition on the value of $u$ of the same form as above in order to ensure that $W_{2}(\infty) \not= \infty$.

Let $\delta_{1} = \alpha - x$ and $\delta_{2} = \beta - z$. We can now apply Lemma \ref{approxlem} to see that $$|\delta_{1}| = |x-\alpha| = |x-M_{m,x}W_{1}(\infty)| < \epsilon.$$ By the same logic, we have $|\delta_{2}|<\epsilon$. Since we know $\alpha$ and $\beta$ are within $\epsilon$ of $x$ and $z$, respectively, we can be sure that $\alpha$ and $\beta$ are both well-defined complex numbers and not infinite.

To conclude the proof, we shall evaluate $\Phi(A)$ in two different ways. Note that since $A(\infty) = \alpha \not = \infty$, we are in the first case of \eqref{phi}. If we apply $\Phi$ to $A$ directly, we get
\begin{align*}
    \Phi(A) &= E_2(0)I(A(\infty)-A^{-1}(\infty))-D(A(\infty)) \\
    &= E_2(0)I(\alpha-\beta)-D(\alpha) \\
    &= E_2(0)I(x+\delta_{1}-z-\delta_{2})-D(\alpha).
\end{align*}
If we apply $\Phi$ to $A$ by breaking it up into smaller matrices in $GL_2(\mathcal{O})$, we see that 
\begin{align*}
\Phi(A) &= \Phi(M_{m, x}T^{u}S^{*}T^{-u}M_{n,z}^{-1})\\
&= \Phi(M_{m, x}) + \Phi(T^{u}) + \Phi(M_{m, x}^{-1}M_{n, z}) + \Phi(T^{-u}) + \Phi(M_{n,z}^{-1})\\
&
= 0.
\end{align*}
Therefore, we have 
$$E_2(0)I(x+\delta_{1}-z-\delta_{2})-D(\alpha) = 0,$$
which gives us $D(\alpha) = E_2(0)I(x+\delta_{1}-z-\delta_{2})$ and hence
    \[
    \tilde{D}(\alpha) = (\sqrt{|d|}i)^{-1}I(x+\delta_{1}-z-\delta_{2}),
    \]
as required, which we note is independent of the choice of representation of $\alpha$. 

\subsection*{Acknowledgments}

The authors thank Daniel Martin for communications regarding their paper. This research was conducted at the REU site: Mathematical Analysis and Applications at the University of Michigan-Dearborn. 

\subsection*{Declarations}
 
\subsubsection*{Ethical Approval }
Not applicable.
 
\subsubsection*{Funding}
 This work was partially supported the NSF grant DMS-1950102 and NSA grant H98230-19. The last author was also supported by NSF grant DMS-2212924.
 
\subsubsection*{Availability of data and materials }
Not applicable.


\bibliographystyle{plain}
\bibliography{sources}
\end{document}